\documentclass[11pt]{article}
\usepackage{a4}
\usepackage{amssymb}
\usepackage{amsmath}
\usepackage{amsthm}
\usepackage{enumerate}
\usepackage[pagebackref,colorlinks,citecolor=blue,linkcolor=blue]{hyperref}

\usepackage{geometry}
\numberwithin{equation}{section}
\theoremstyle{plain}
\newtheorem{theorem}[equation]{Theorem}
\newtheorem{corollary}[equation]{Corollary}
\newtheorem{lemma}[equation]{Lemma}
\newtheorem{proposition}[equation]{Proposition}

\theoremstyle{definition}
\newtheorem{definition}[equation]{Definition}

\newtheorem{example}[equation]{Example}
\newtheorem{remark}[equation]{Remark}

\numberwithin{equation}{section}

\newcommand{\R}{{\mathbb R}}
\newcommand{\N}{{\mathbb N}}

\newcommand{\Om}{\Omega}

\providecommand{\vint}[1]{\mathchoice
          {\mathop{\vrule width 5pt height 3 pt depth -2.5pt
                  \kern -9pt \kern 1pt\intop}\nolimits_{\kern -5pt{#1}}}
          {\mathop{\vrule width 5pt height 3 pt depth -2.6pt
                  \kern -6pt \intop}\nolimits_{\kern -3pt{#1}}}
          {\mathop{\vrule width 5pt height 3 pt depth -2.6pt
                  \kern -6pt \intop}\nolimits_{\kern -3pt{#1}}}
          {\mathop{\vrule width 5pt height 3 pt depth -2.6pt
                  \kern -6pt \intop}\nolimits_{\kern -3pt{#1}}}}

\newcommand{\eps}{\varepsilon}
\newcommand{\loc}{\mathrm{loc}}

\newcommand{\BV}{\mathrm{BV}}
\newcommand{\SBV}{\mathrm{SBV}}

\newcommand{\liploc}{\mathrm{Lip}_{\mathrm{loc}}}

\newcommand{\ch}{\text{\raise 1.3pt \hbox{$\chi$}\kern-0.2pt}}

\DeclareMathOperator{\capa}{Cap}
\DeclareMathOperator{\rcapa}{cap}

\DeclareMathOperator{\supp}{spt}

\DeclareMathOperator{\fint}{fine-int}

\begin{document}
\title{Discrete convolutions of $\BV$ functions\\
in quasiopen sets in metric spaces
\footnote{
\hfill \break {\it Keywords\,}: metric measure space, function of bounded variation,
discrete convolution, quasiopen, partition of unity, uniform approximation
}}
\author{Panu Lahti}
\maketitle

\begin{abstract}
We study fine potential theory and in particular
partitions of unity in quasiopen sets in the case $p=1$.
Using these, we develop an analog of the discrete convolution technique
in quasiopen (instead of open) sets.
We apply this technique to show that every
function of bounded variation ($\BV$ function)
can be approximated in the $\BV$ and
$L^{\infty}$ norms by
$\BV$ functions whose jump sets are of finite Hausdorff measure.
Our results seem to be new even in Euclidean spaces but we work in
a more general complete metric space that is equipped with a doubling measure and supports a
Poincar\'e inequality. 
\end{abstract}

\section{Introduction}

In Euclidean spaces, a standard and very useful
method for approximating a function of bounded variation ($\BV$ function)
by smooth functions in a weak sense
is to take convolutions
with mollifier functions.
In the setting of a more general doubling metric measure space, an analog of this method
is given by so-called discrete convolutions. These are constructed by means of
Lipschitz partitions of unity subordinate to Whitney coverings of an open set,
and they possess most of the good properties of standard convolutions.
Discrete convolutions and their properties have been considered e.g. in
\cite{HKT,KKST3,LaSh}. Whitney coverings and related partitions of unity were
originally developed in \cite{CoWe,MaSe,Wh}.

In open sets, it is of course easy to pick Lipschitz cutoff functions that
are then used in constructing a partition of unity. On the other hand, being limited to open sets
is also a drawback of (discrete) convolutions; sometimes one may wish
to smooth out a function in a finer way. In potential theory, one sometimes works
with the concept of \emph{quasiopen} sets.
For nonlinear potential theory and its history in the Euclidean setting,
in the case $1<p<\infty$, see especially the monographs \cite{AH,HKM,MZ}.
Nonlinear fine potential theory in metric spaces has been
studied in several papers in recent years, see \cite{BBL-SS,BBL-CCK,BBL-WC}.
The typical assumptions on a metric space, which we make also in the current paper, 
are that the space is complete, equipped with a doubling measure, and supports
a Poincar\'e inequality; see Section \ref{sec:preliminaries} for definitions.

Much less is known (even in Euclidean spaces) in the case $p=1$,
but certain results of fine potential theory when $p=1$
have been developed by the author in metric spaces
in \cite{L-Fed,L-NC,L-FC}.
In quasiopen sets, the role of Lipschitz cutoff functions 
needs to be taken by Sobolev functions (often called Newton-Sobolev functions in
metric spaces). A theory of Newton-Sobolev cutoff
functions in quasiopen sets when $p=1$ was developed in \cite{L-NC},
analogously to the case $1<p<\infty$ studied previously in \cite{BBL-SS}.
In the current paper we apply this theory
to construct partitions of unity in quasiopen sets, and then we develop an analog
of the discrete convolution technique in such sets.
This is given in Theorem
\ref{thm:discrete conv in quasiopen set} and is, as far
as we know, new even in Euclidean spaces.

As an application, we prove a new approximation result for $\BV$
functions. The \emph{jump set} of a $\BV$ function is always $\sigma$-finite, but
not necessarily finite, with respect to the codimension one (in the Euclidean setting,
$n-1$-dimensional) Hausdorff measure. On the other hand, in the study of minimization problems one often considers subclasses of $\BV$ functions for which the jump
set is of finite Hausdorff measure. Approximation results for this kind of $\BV$
 functions by means of piecewise smooth functions were studied recently in \cite{dPFP}.
 In the current paper, we prove that it is possible
 to approximate an arbitrary $\BV$ function by $\BV$ functions whose jump sets are
 of finite Hausdorff measure, in the following sense.
 
\begin{theorem}\label{thm:main approximation theorem}
Let $\Om\subset X$ be an open set and let $u\in\BV(\Om)$. Then
there exists a sequence $(u_i)\subset \BV(\Om)$ such that
$\Vert u_i- u\Vert_{\BV(\Om)}+\Vert u_i- u\Vert_{L^{\infty}(\Om)}\to 0$,
and $\mathcal H(S_{u_i})<\infty$
for each $i\in\N$.
\end{theorem}
 
 This is given (with more details)
 in Theorem \ref{thm:approximation}.
 Note that here the approximation is not only in the usual weak sense but
 in the $\BV$ \emph{norm}. Yet the most subtle problem seems to be
 to obtain approximation simultaneously in the $L^{\infty}$ norm;
 for this the usual (discrete) convolution method seems too crude,
 demonstrating the need for the ``quasiopen version''.

\section{Definitions and assumptions}\label{sec:preliminaries}

In this section we present the notation, definitions,
and assumptions used in the paper.

Throughout the paper, $(X,d,\mu)$ is a complete metric space that is equip\-ped
with a metric $d$ and a Borel regular outer measure $\mu$ satisfying
a doubling property, meaning that
there exists a constant $C_d\ge 1$ such that
\[
0<\mu(B(x,2r))\le C_d\mu(B(x,r))<\infty
\]
for every ball $B(x,r):=\{y\in X:\,d(y,x)<r\}$.
When we want to state that a constant $C$
depends on the parameters $a,b, \ldots$, we write $C=C(a,b,\ldots)$.
When a property holds outside a set of $\mu$-measure zero, we say that it holds
almost everywhere, abbreviated a.e.

All functions defined on $X$ or its subsets will take values in $[-\infty,\infty]$.
As a complete metric space equipped with a doubling measure, $X$ is proper,
that is, every closed and bounded set is compact.
Given a $\mu$-measurable set $A\subset X$, we define $L^1_{\loc}(A)$ as the class
of functions $u$ on $A$
such that for every $x\in A$ there exists $r>0$ such that $u\in L^1(A\cap B(x,r))$.
Other local spaces of functions are defined similarly.
For an open set $\Omega\subset X$,
a function is in the class $L^1_{\loc}(\Omega)$ if and only if it is in $L^1(\Om')$ for
every open $\Omega'\Subset\Omega$.
Here $\Omega'\Subset\Omega$ means that $\overline{\Omega'}$ is a
compact subset of $\Omega$.

For any set $A\subset X$ and $0<R<\infty$, the restricted Hausdorff content
of codimension one is defined by
\[
\mathcal{H}_{R}(A):=\inf\left\{ \sum_{i\in I}
\frac{\mu(B(x_{i},r_{i}))}{r_{i}}:\,A\subset\bigcup_{i\in I}B(x_{i},r_{i}),\,r_{i}\le R\right\},
\]
where $I\subset \N$ is a finite or countable index set.
The codimension one Hausdorff measure of $A\subset X$ is then defined by
\[
\mathcal{H}(A):=\lim_{R\rightarrow 0}\mathcal{H}_{R}(A).
\]

By a curve we mean a rectifiable continuous mapping from a compact interval of the real line into $X$.
A nonnegative Borel function $g$ on $X$ is an upper gradient 
of a function $u$
on $X$ if for all nonconstant curves $\gamma$, we have
\begin{equation}\label{eq:definition of upper gradient}
|u(x)-u(y)|\le \int_{\gamma} g \,ds,
\end{equation}
where $x$ and $y$ are the end points of $\gamma$ and the curve integral
is defined by means of an arc-length parametrization, see \cite[Section 2]{HK}
where upper gradients were originally introduced.
We interpret $|u(x)-u(y)|=\infty$ whenever  
at least one of $|u(x)|$, $|u(y)|$ is infinite.

We say that a family of curves $\Gamma$ is of zero $1$-modulus if there is a 
nonnegative Borel function $\rho\in L^1(X)$ such that 
for all curves $\gamma\in\Gamma$, the curve integral $\int_\gamma \rho\,ds$ is infinite.
A property is said to hold for $1$-almost every curve
if it fails only for a curve family with zero $1$-modulus. 
If $g$ is a nonnegative $\mu$-measurable function on $X$
and (\ref{eq:definition of upper gradient}) holds for $1$-almost every curve,
we say that $g$ is a $1$-weak upper gradient of $u$. 
By only considering curves $\gamma$ in a set $A\subset X$,
we can talk about a function $g$ being a ($1$-weak) upper gradient of $u$ in $A$.

For a $\mu$-measurable set $H\subset X$, we define
\[
\Vert u\Vert_{N^{1,1}(H)}:=\Vert u\Vert_{L^1(H)}+\inf \Vert g\Vert_{L^1(H)},
\]
where the infimum is taken over all $1$-weak upper gradients $g$ of $u$ in $H$.
The substitute for the Sobolev space $W^{1,1}$ in the metric setting is the Newton-Sobolev space
\[
N^{1,1}(H):=\{u:\,\|u\|_{N^{1,1}(H)}<\infty\},
\]
which was first introduced in \cite{S}.
We also define the Dirichlet space $D^1(H)$ consisting of $\mu$-measurable
functions $u$ on $H$ with an upper gradient $g\in L^1(H)$ in $H$.
Both spaces are clearly vector spaces and
by \cite[Corollary 1.20]{BB} (or its proof) we know that each is also a lattice,
so that
\begin{equation}\label{eq:Newton Sobolev functions form lattice}
\textrm{if }u,v\in D^1(X),
\textrm{ then }\min\{u,v\},\,\max\{u,v\}\in  D^1(X).
\end{equation}
For any $H\subset X$, the space of Newton-Sobolev functions with zero boundary values is defined as
\[
N_0^{1,1}(H):=\{u|_{H}:\,u\in N^{1,1}(X)\textrm{ and }u=0\textrm { on }X\setminus H\}.
\]
This space is a subspace of $N^{1,1}(H)$ when $H$ is $\mu$-measurable, and it can always
be understood to be a subspace of $N^{1,1}(X)$.
The class $N_c^{1,1}(H)$ consists of those functions $u\in N^{1,1}(X)$ that have
compact support in $H$, i.e. $\supp u\subset H$.

Note that we understand Newton-Sobolev functions to be defined at every $x\in H$
(even though $\Vert \cdot\Vert_{N^{1,1}(H)}$ is then only a seminorm).
It is known that for any $u\in N_{\loc}^{1,1}(H)$ there exists a minimal $1$-weak
upper gradient of $u$ in $H$, always denoted by $g_{u}$, satisfying $g_{u}\le g$ 
a.e. in $H$, for any $1$-weak upper gradient $g\in L_{\loc}^{1}(H)$
of $u$ in $H$, see \cite[Theorem 2.25]{BB}. Sometimes we also
use the notation $g_{u,H}$ to specify that we mean the minimal $1$-weak
upper gradient of $u$ in $H$, even though $u$ may be defined in a larger set.

We will assume throughout the paper that $X$ supports a $(1,1)$-Poincar\'e inequality,
meaning that there exist constants $C_P>0$ and $\lambda \ge 1$ such that for every
ball $B(x,r)$, every $u\in L^1_{\loc}(X)$,
and every upper gradient $g$ of $u$,
we have
\begin{equation}\label{eq:poincare}
\vint{B(x,r)}|u-u_{B(x,r)}|\, d\mu 
\le C_P r\vint{B(x,\lambda r)}g\,d\mu,
\end{equation}
where 
\[
u_{B(x,r)}:=\vint{B(x,r)}u\,d\mu :=\frac 1{\mu(B(x,r))}\int_{B(x,r)}u\,d\mu.
\]

The $1$-capacity of a set $A\subset X$ is defined by
\[
\capa_1(A):=\inf \Vert u\Vert_{N^{1,1}(X)},
\]
where the infimum is taken over all functions $u\in N^{1,1}(X)$ such that $u\ge 1$ in $A$.
If a property holds outside a set
$A\subset X$ with $\capa_1(A)=0$, we say that it holds $1$-quasieverywhere, or $1$-q.e.
We know that for any $\mu$-measurable set $H\subset X$,
\begin{equation}\label{eq:quasieverywhere equivalence classes}
u=0\ \textrm{ 1-q.e. in }H\textrm{ implies }\ \Vert u\Vert_{N^{1,1}(H)}=0,
\end{equation}
see \cite[Proposition 1.61]{BB}.

The variational $1$-capacity of a set $A\subset H$ with respect to a set $H\subset X$
is defined by
\[
\rcapa_1(A,H):=\inf \int_X g_u\,d\mu,
\]
where the infimum is taken over functions $u\in N^{1,1}_0(H)$ such that
$u\ge 1$ in $A$, and where $g_u$ is the minimal $1$-weak upper gradient of $u$ (in $X$).
By truncation, we can alternatively require that $u= 1$ in $A$.
For basic properties satisfied by capacities, such as monotonicity and countable subadditivity, see e.g. \cite{BB}.
By \cite[Theorem 4.3, Theorem 5.1]{HaKi} we know that for $A\subset X$,
\begin{equation}\label{eq:null sets of Hausdorff measure and capacity}
\capa_{1}(A)=0\quad
\textrm{if and only if}\quad\mathcal H(A)=0.
\end{equation}

We say that a set $U\subset X$ is $1$-quasiopen
if for every $\eps>0$ there is an
open set $G\subset X$ such that $\capa_1(G)<\eps$ and $U\cup G$ is open.
Given a set $H\subset X$, we say that a function $u$
is $1$-quasi (lower/upper semi-)continuous on $H$ if
for every $\eps>0$ there is an open set $G\subset X$ such that $\capa_1(G)<\eps$
and $u|_{H\setminus G}$ is finite and (lower/upper semi-)continuous.
If $H=X$, we do not mention it.
It is a well-known fact that Newton-Sobolev functions are $1$-quasicontinuous
on open sets,
see \cite[Theorem 1.1]{BBS} or \cite[Theorem 5.29]{BB}.
In fact, by \cite[Theorem 1.3]{BBM} we know  that more generally
\begin{equation}\label{eq:quasicontinuous on quasiopen set}
\textrm{for a 1-quasiopen }U\subset X,\textrm{ any }u\in N^{1,1}_{\loc}(U)\textrm{ is 1-quasicontinuous on }U.
\end{equation}
By \cite[Proposition 5.23]{BB} we also know that for a $1$-quasiopen
$U\subset X$ and functions $u,v$ that are $1$-quasicontinuous on $U$,
\begin{equation}\label{eq:ae to qe}
\textrm{if }u=v\textrm{ a.e. in }U,\textrm{ then }u=v\textrm{ 1-q.e. in }U.
\end{equation}
More precisely, this result is given with respect to a version of $\capa_1$ defined by
considering $U$ as the metric space, but \cite[Proposition 4.2]{BBM}
and \cite[Remark 3.5]{S2} guarantee that this does not make a difference.

Next we present the definition and basic properties of functions
of bounded variation on metric spaces, following \cite{M}. See also e.g.
the monographs \cite{AFP, EvaG92, Fed, Giu84, Zie89} for the classical 
theory in the Euclidean setting.
Given an open set $\Om\subset X$ and a function $u\in L^1_{\loc}(\Om)$,
we define the total variation of $u$ in $\Om$ by
\begin{equation}\label{eq:definition of Du in open set}
\|Du\|(\Om):=\inf\left\{\liminf_{i\to\infty}\int_\Om g_{u_i}\,d\mu:\, u_i\in N^{1,1}_{\loc}(\Om),\, u_i\to u\textrm{ in } L^1_{\loc}(\Om)\right\},
\end{equation}
where each $g_{u_i}$ is the minimal $1$-weak upper gradient of $u_i$
in $\Om$.
(In \cite{M}, local Lipschitz constants were used in place of upper gradients, but the theory
can be developed similarly with either definition.)
We say that a function $u\in L^1(\Om)$ is of bounded variation, 
and denote $u\in\BV(\Om)$, if $\|Du\|(\Om)<\infty$.
For an arbitrary set $A\subset X$, we define
\begin{equation}\label{eq:definition of Du in arbitrary set}
\|Du\|(A):=\inf\{\|Du\|(W):\, A\subset W,\,W\subset X
\text{ is open}\}.
\end{equation}

Note that if we defined $\Vert Du\Vert(A)$ simply by replacing $\Om$
with $A$ in \eqref{eq:definition of Du in open set}, we would get a different
quantity compared with the definition given in
\eqref{eq:definition of Du in arbitrary set}.
However, in a $1$-quasiopen set $U$ these give the same result;
we understand the expression $\Vert Du\Vert(U)<\infty$ to mean that
there exists some open set $\Om\supset U$
such that
$u\in L^1_{\loc}(\Om)$ and $\Vert Du\Vert(\Om)<\infty$.

\begin{theorem}[{\cite[Theorem 4.3]{L-LSC}}]\label{thm:characterization of total variational}
	Let $U\subset X$ be $1$-quasiopen. If $\Vert Du\Vert(U)<\infty$, then
	\[
	\Vert Du\Vert(U)=\inf \left\{\liminf_{i\to\infty}\int_{U}g_{u_i}\,d\mu,\,
	u_i\in N_{\loc}^{1,1}(U),\, u_i\to u\textrm{ in }L^1_{\loc}(U)\right\},
	\]
	where each $g_{u_i}$ is the minimal $1$-weak upper gradient of $u_i$ in $U$.
\end{theorem}

Note that $1$-quasiopen sets are $\mu$-measurable
by \cite[Lemma 9.3]{BB-OD}. We also have the following lower semicontinuity.

\begin{theorem}[{\cite[Theorem 4.5]{L-LSC}}]\label{thm:lower semic in quasiopen sets}
	Let $U\subset X$ be a $1$-quasiopen set.
	If $\Vert Du\Vert(U)<\infty$ and $u_i\to u$ in $L^1_{\loc}(U)$, then
	\[
	\Vert Du\Vert(U)\le \liminf_{i\to\infty}\Vert Du_i\Vert(U).
	\]
\end{theorem}

If $u\in L^1_{\loc}(\Om)$ and $\Vert Du\Vert(\Omega)<\infty$,
then $\|Du\|$ is
a Radon measure on $\Omega$ by \cite[Theorem 3.4]{M}, and we call it the variation
measure.
The $\BV$ norm is defined by
\[
\Vert u\Vert_{\BV(\Om)}:=\Vert u\Vert_{L^1(\Om)}+\Vert Du\Vert(\Om).
\]
A $\mu$-measurable set $E\subset X$ is said to be of finite perimeter if $\|D\ch_E\|(X)<\infty$, where $\ch_E$ is the characteristic function of $E$.
The measure-theoretic interior of a set $E\subset X$ is defined by
\begin{equation}\label{eq:measure theoretic interior}
I_E:=
\left\{x\in X:\,\lim_{r\to 0}\frac{\mu(B(x,r)\setminus E)}{\mu(B(x,r))}=0\right\},
\end{equation}
and the measure-theoretic exterior by
\[
O_E:=
\left\{x\in X:\,\lim_{r\to 0}\frac{\mu(B(x,r)\cap E)}{\mu(B(x,r))}=0\right\}.
\]
The measure-theoretic boundary $\partial^{*}E$ is defined as the set of points
$x\in X$
at which both $E$ and its complement have strictly positive upper density, i.e.
\[
\limsup_{r\to 0}\frac{\mu(B(x,r)\cap E)}{\mu(B(x,r))}>0\quad
\textrm{and}\quad\limsup_{r\to 0}\frac{\mu(B(x,r)\setminus E)}{\mu(B(x,r))}>0.
\]
For an open set $\Omega\subset X$ and a $\mu$-measurable set $E\subset X$ with
$\Vert D\ch_E\Vert(\Om)<\infty$, we know that for any Borel set $A\subset\Omega$,
\begin{equation}\label{eq:def of theta}
\Vert D\ch_E\Vert(A)=\int_{\partial^{*}E\cap A}\theta_E\,d\mathcal H,
\end{equation}
where
$\theta_E\colon X\to [\alpha,C_d]$ with $\alpha=\alpha(C_d,C_P,\lambda)>0$, see \cite[Theorem 5.3]{A1} 
and \cite[Theorem 4.6]{AMP}.

For any $u,v\in L^1_{\loc}(\Om)$ and any $A\subset \Om$,
it is straightforward to show that
\begin{equation}\label{eq:BV functions form vector space}
\Vert D(u+v)\Vert(A)\le \Vert Du\Vert(A)+\Vert Dv\Vert(A).
\end{equation}

The lower and upper approximate limits of a function $u$ on an open set $\Om$
are defined respectively by
\[
u^{\wedge}(x):
=\sup\left\{t\in\R:\,\lim_{r\to 0}\frac{\mu(B(x,r)\cap\{u<t\})}{\mu(B(x,r))}=0\right\}
\]
and
\[
u^{\vee}(x):
=\inf\left\{t\in\R:\,\lim_{r\to 0}\frac{\mu(B(x,r)\cap\{u>t\})}{\mu(B(x,r))}=0\right\}
\]
for $x\in \Om$. Always $u^{\wedge}\le u^{\vee}$,
and the jump set of $u$ is defined by
\[
S_u:=\{u^{\wedge}<u^{\vee}\}:=\{x\in\Om:\,u^{\wedge}(x)<u^{\vee}(x)\}.
\]
Note that since we understand $u^{\wedge}$ and $u^{\vee}$ to be defined only on $\Om$,
also $S_u$ is understood to be a subset of $\Om$.
For $u\in L^1_{\loc}(\Om)$, we have $u=u^{\wedge}=u^{\vee}$ a.e. in $\Om$
by Lebesgue's differentiation theorem (see e.g. \cite[Chapter 1]{Hei}).
Unlike Newton-Sobolev functions, we understand $\BV$ functions to be
$\mu$-equivalence classes. To consider fine properties, we need to
consider the pointwise representatives $u^{\wedge}$ and $u^{\vee}$.

Recall that Newton-Sobolev functions are quasicontinuous;
$\BV$ functions have the following
quasi-semicontinuity property, which follows from \cite[Corollary 4.2]{L-SA},
which in turn is based on \cite[Theorem 1.1]{LaSh}.
The property was first proved in the Euclidean setting in
\cite[Theorem 2.5]{CDLP}.

\begin{proposition}\label{prop:quasisemicontinuity}
Let $\Om\subset X$ be open and let
$u\in L^1_{\loc}(\Om)$ with $\Vert Du\Vert(\Om)<\infty$. Then
$u^{\wedge}$ is $1$-quasi lower semicontinuous and
$u^{\vee}$ is $1$-quasi upper semicontinuous on $\Om$.
\end{proposition}

By \cite[Theorem 5.3]{AMP}, the variation measure of a $\BV$ function
can be decomposed into the absolutely continuous and singular part, and the latter
into the Cantor and jump part, as follows. Given an open set 
$\Omega\subset X$ and $u\in L^1_{\loc}(\Omega)$
with $\Vert Du\Vert(\Om)<\infty$, we have for any Borel set $A\subset \Om$
\begin{equation}\label{eq:variation measure decomposition}
\begin{split}
\Vert Du\Vert(A) &=\Vert Du\Vert^a(A)+\Vert Du\Vert^s(A)\\
&=\Vert Du\Vert^a(A)+\Vert Du\Vert^c(A)+\Vert Du\Vert^j(A)\\
&=\int_{A}a\,d\mu+\Vert Du\Vert^c(A)+\int_{A\cap S_u}\int_{u^{\wedge}(x)}^{u^{\vee}(x)}\theta_{\{u>t\}}(x)\,dt\,d\mathcal H(x),
\end{split}
\end{equation}
where $a\in L^1(\Omega)$ is the density of the absolutely continuous part
and the functions $\theta_{\{u>t\}}\in [\alpha,C_d]$ 
are as in~\eqref{eq:def of theta}.
In \cite{AMP} it is assumed that $u\in\BV(\Om)$, but the proof is the same
for the slightly more general $u$ that we consider here.

Next we define the fine topology in the case $p=1$.
\begin{definition}\label{def:1 fine topology}
We say that $A\subset X$ is $1$-thin at the point $x\in X$ if
\[
\lim_{r\to 0}r\frac{\rcapa_1(A\cap B(x,r),B(x,2r))}{\mu(B(x,r))}=0.
\]
We also say that a set $U\subset X$ is $1$-finely open if $X\setminus U$ is $1$-thin at every $x\in U$. Then we define the $1$-fine topology as the collection of $1$-finely open sets on $X$.

We denote the $1$-fine interior of a set $H\subset X$, i.e. the largest $1$-finely open set contained in $H$, by $\fint H$. We denote the $1$-fine closure of $H$, i.e. the smallest $1$-finely closed set containing $H$, by $\overline{H}^1$.

We say that a function $u$ defined on a set $U\subset X$ is $1$-finely continuous at $x\in U$ if it is continuous at $x$ when $U$ is equipped with the induced $1$-fine topology on $U$ and $[-\infty,\infty]$ is equipped
with the usual topology.
\end{definition}

See \cite[Section 4]{L-FC} for discussion on this definition, and for a proof of the fact that the $1$-fine topology is indeed a topology.
By \cite[Lemma 3.1]{L-Fed}, $1$-thinness implies zero measure density, i.e.
\begin{equation}\label{eq:thinness and measure thinness}
\textrm{if }A\textrm{ is 1-thin at }x,\ \textrm{ then }\lim_{r\to 0}\frac{\mu(A\cap B(x,r))}{\mu(B(x,r))}=0.
\end{equation}

\begin{theorem}[{\cite[Corollary 6.12]{L-CK}}]\label{thm:finely open is quasiopen and vice versa}
A set $U\subset X$ is $1$-quasiopen if and only if it is the union of a $1$-finely
open set and a $\mathcal H$-negligible set.
\end{theorem}

\begin{theorem}[{\cite[Theorem 5.1]{L-NC}}]\label{thm:fine continuity and quasicontinuity equivalence}
	A function $u$ on a $1$-quasiopen set $U$ is
	$1$-quasicontinuous on $U$ if and only if it is finite $1$-q.e. and $1$-finely
	continuous $1$-q.e. in $U$.
\end{theorem}

\emph{Throughout this paper we assume that $(X,d,\mu)$ is a complete metric space
	that is equipped with a doubling measure $\mu$ and supports a
	$(1,1)$-Poincar\'e inequality.}

\section{Preliminary results}\label{sec:preliminary results}

In this section we prove and record some preliminary results 
needed in constructing the discrete convolutions in $1$-quasiopen sets.
We start with simple lemmas concerning the total variation.
The first lemma states that in the definition of the total variation, we can
consider convergence in $L^1(\Om)$
instead of convergence in  $L_{\loc}^1(\Om)$.

\begin{lemma}[{\cite[Lemma 5.5]{KLLS}}]\label{lem:L1 loc and L1 convergence}
Let $\Omega\subset X$ be an open set and let $u\in L^1_{\loc}(\Omega)$
with $\Vert Du\Vert(\Omega)<\infty$. Then there exists a sequence
$(u_i)\subset \liploc(\Omega)$ with $u_i-u\to 0$ in $L^1(\Omega)$ and 
$\int_\Omega g_{u_i}\, d\mu\to \Vert Du\Vert(\Omega)$, where each
$g_{u_i}$ is the minimal $1$-weak upper gradient of $u_i$ in $\Om$.
\end{lemma}

Note that we cannot write $u_i\to u$ in $L^1(\Omega)$, since the functions $u_i,u$ are not necessarily
in the class $L^1(\Om)$.

Now we generalize this to $1$-quasiopen sets.

\begin{lemma}\label{lem:optimal sequence in quasiopen sets}
Let $U\subset \Om\subset X$ be such that $U$ is $1$-quasiopen and
$\Om$ is open, and let
$u\in L^1_{\loc}(\Om)$ with $\Vert Du\Vert(\Om)<\infty$.
Then there exists a sequence
	$(u_i)\subset \liploc(U)$ such that $u_i- u\to 0$ in $L^1(U)$ and
	\[
	\lim_{i\to \infty}\int_{U}g_{u_i,U}\,d\mu=\Vert Du\Vert(U).
	\]
\end{lemma}
Recall that $g_{u_i,U}$ denotes the minimal $1$-weak upper gradient of $u_i$ in $U$.
\begin{proof}
Take
open sets $\Om_i$ such that $U\subset \Om_i\subset \Om$ and
$\Vert Du\Vert(\Om_i)<\Vert Du\Vert(U)+1/i$, for each $i\in\N$.
By Lemma \ref{lem:L1 loc and L1 convergence} we find functions
$u_i\in \liploc(\Om_i)\subset \liploc(U)$ such that
$\Vert u_i-u\Vert_{L^1(\Om_i)}<1/i$ and
\[
\int_{\Om_i}g_{u_i,\Om_i}\,d\mu< \Vert Du\Vert(\Om_i)+1/i.
\]
It follows that $u_i- u\to 0$ in $L^1(U)$ and
	\[
	\limsup_{i\to \infty}\int_{U}g_{u_i,U}\,d\mu
	\le \limsup_{i\to \infty}\int_{U}g_{u_i,\Om_i}\,d\mu
	\le\Vert Du\Vert(U),
	\]
and then by Theorem \ref{thm:characterization of total variational} we must
in fact have
\[
	\lim_{i\to \infty}\int_{U}g_{u_i,U}\,d\mu=\Vert Du\Vert(U).
	\]
\end{proof}

Next we consider preliminary approximation results for $\BV$ functions.
We have the following approximation result for $\BV$ functions whose jumps
remain bounded.

\begin{proposition}[{\cite[Proposition 5.2]{L-Appr}}]\label{prop:uniform approximation}
	Let $U\subset \Om$ such that $U$ is $1$-quasiopen
	and $\Om$ is open, and let $u\in\BV(\Om)$ and $\beta>0$ such that
	$u^{\vee}-u^{\wedge}<\beta$ in $U$. Then for every $\eps>0$ there exists
	$v\in N^{1,1}(U)$ such that
	$\Vert v-u\Vert_{L^{\infty}(U)}\le 4\beta$ and
	\[
	\int_{U}g_v\,d\mu< \Vert Du\Vert(U)+\eps.
	\]
\end{proposition}

In fact, in the proof of the above proposition in \cite{L-Appr}, the $L^{\infty}$-bound
is stated in the following slightly more precise way (note that $v$, $u^{\wedge}$,
and $u^{\vee}$ are all pointwise defined functions):
\begin{equation}\label{eq:v and u pointwise representative L infty bound}
u^{\vee}-4\beta \le v \le u^{\wedge}+4\beta\quad\textrm{in }U.
\end{equation}

By \cite[Corollary 2.21]{BB} we know that if $H\subset X$ is a
$\mu$-measurable set and $v,w\in N_{\loc}^{1,1}(H)$, then
\begin{equation}\label{eq:upper gradient in coincidence set}
g_v=g_w\ \ \textrm{a.e. in}\ \{x\in H:\,v(x)=w(x)\},
\end{equation}
where $g_v$ and $g_w$ are the minimal $1$-weak upper gradients of $v$ and $w$ in $H$.

The following proposition improves on
Lemma \ref{lem:optimal sequence in quasiopen sets}
by adding an $L^{\infty}$-bound.

\begin{proposition}\label{prop:approximation with L infinity bound}
Let $U\subset \Om\subset X$ be such that $U$ is $1$-quasiopen and
$\Om$ is open, and let $u\in\BV(\Om)$ and $\beta>0$ such that
	$u^{\vee}-u^{\wedge}<\beta$ in $U$.
Then there exists a sequence
$(u_i)\subset N^{1,1}(U)$ such that
$u_i\to u$ in $L^1(U)$, $\sup_{U}| u_i-u^{\vee}|\le 9\beta$
for all $i\in\N$, and
\[
\lim_{i\to\infty}\int_{U} g_{u_i}\,d\mu=\Vert Du\Vert(U),
\]
where each $g_{u_i}$
is the minimal $1$-weak upper gradient of $u_i$ in $U$.
\end{proposition}

Also in the proof below, $g$ with a subscript always denotes the minimal
$1$-weak upper gradient of a function in $U$
(even though we sometimes integrate it only over a subset of $U$).
The proof reveals that we also have $\sup_{U}| u_i-u^{\wedge}|\le 9\beta$,
that is, we can replace the pointwise representative $u^{\vee}$ by
$u^{\wedge}$.

\begin{proof}
By Proposition \ref{prop:uniform approximation}
and \eqref{eq:v and u pointwise representative L infty bound} we find a function
$v\in N^{1,1}(U)$ such that
	$u^{\vee}-4\beta \le v \le u^{\wedge}+4\beta$ in $U$ and
	\[
	\int_{U}g_v\,d\mu\le \Vert Du\Vert(U)+1.
	\]
	Define $v_1:=v-5\beta$ and $v_2:=v+5\beta$, so that
	$v_1,v_2\in N_{\loc}^{1,1}(U)$ with
	$u^{\vee} -9\beta\le v_1\le u^{\wedge}-\beta$
	and $u^{\vee}+\beta\le v_2\le u^{\wedge}+9\beta$ in $U$,
	and
	\begin{equation}\label{eq:control of gvj}
	\int_{U}g_{v_j}\,d\mu\le \Vert Du\Vert(U)+1
	\end{equation}
	for $j=1,2$.
	By Lemma \ref{lem:optimal sequence in quasiopen sets} we find a sequence
	$(w_i)\subset N^{1,1}(U)$ such that $w_i\to u$ in $L^1(U)$ and
	\[
	\lim_{i\to \infty}\int_{U}g_{w_i}\,d\mu=\Vert Du\Vert(U).
	\]
	By passing to a subsequence (not relabeled), we can assume that also
	$w_i\to u$ a.e. in $U$.
	Then define $u_i:=\min\{v_2,\max\{v_1, w_i\}\}$.
	By Lebesgue's differentiation theorem,
	we have also $u -9\beta\le v_1\le u-\beta$
	and $u+\beta\le v_2\le u+9\beta$ a.e. in $U$, whence
	\[
	\Vert u_i-u\Vert_{L^{1}(U)}\le \Vert w_i-u\Vert_{L^{1}(U)}\to 0,
	\]
	that is $u_i\to u$ in $L^1(U)$, as desired.
	Moreover, $\sup_{U}| u_i-u^{\vee}|\le 9\beta$ for all $i\in\N$.
	In addition, by \eqref{eq:upper gradient in coincidence set} we have
	for each $i\in\N$
	\begin{align*}
	\int_{U}g_{u_i}\,d\mu
	&=\int_{\{w_i>v_2\}}g_{v_2}\,d\mu
	+\int_{\{w_i<v_1\}}g_{v_1}\,d\mu
	+\int_{\{v_1\le w_i\le v_2\}}g_{w_i}\,d\mu\\
	&\le\int_{\{w_i>v_2\}}g_{v_2}\,d\mu
	+\int_{\{w_i<v_1\}}g_{v_1}\,d\mu
	+\int_{U}g_{w_i}\,d\mu.
	\end{align*}
	Since $\int_{U}g_{v_2}\,d\mu<\infty$ by \eqref{eq:control of gvj}
	and since $w_i\to u<v_2$ a.e. in $U$,
	by Lebesgue's dominated convergence theorem we get 
	$\int_{\{w_i>v_2\}}g_{v_2}\,d\mu\to 0$.
	Treating the integral involving $v_1$ similarly, we get
	\[
	\limsup_{i\to\infty}\int_{U}g_{u_i}\,d\mu
	\le\limsup_{i\to\infty}\int_{U}g_{w_i}\,d\mu=\Vert Du\Vert(U),
	\]
	and then in fact $\lim_{i\to\infty}\int_{U}g_{u_i}\,d\mu=\Vert Du\Vert(U)$
	by Theorem \ref{thm:characterization of total variational}.
\end{proof}

The variation measure is always absolutely continuous with respect to the
$1$-capacity, in the following sense.

\begin{lemma}[{\cite[Lemma 3.8]{L-SA}}]\label{lem:variation measure and capacity}
	Let $\Omega\subset X$ be an open set and
	let $u\in L^1_{\loc}(\Omega)$ with $\Vert Du\Vert(\Omega)<\infty$. Then for every
	$\eps>0$ there exists $\delta>0$ such that if $A\subset \Omega$ with $\capa_1 (A)<\delta$,
	then $\Vert Du\Vert(A)<\eps$.
\end{lemma}

The following proposition describes the weak* convergence of the variation measure;
recall that we understand the expression $\Vert Du\Vert(U)<\infty$ to mean that
there exists some open set $\Om\supset U$
such that
$u\in L^1_{\loc}(\Om)$ and $\Vert Du\Vert(\Om)<\infty$.

\begin{proposition}[{\cite[Proposition 3.9]{L-Leib}}]\label{prop:weak convergence with quasisemicontinuous test function}
	Let $U\subset X$ be $1$-quasiopen.
	If $\Vert Du\Vert(U)<\infty$ and $u_i\to u$ in $L^1_{\loc}(U)$
	such that
	\[
	\Vert Du\Vert(U)= \lim_{i\to\infty}\int_U g_{u_i}\,d\mu,
	\]
	where each $g_{u_i}$ is the minimal $1$-weak upper gradient of $u_i$ in $U$, then
	\[
	\int_U \eta\,d\Vert Du\Vert\ge \limsup_{i\to\infty}\int_U \eta g_{u_i}\,d\mu
	\]
	for every nonnegative bounded $1$-quasi upper semicontinuous function $\eta$ on $U$.
\end{proposition}

Recall the definition of $1$-quasi upper semicontinuity:
for every $\eps>0$ there is an open set $G\subset X$ such that $\capa_1(G)<\eps$
and $\eta|_{U\setminus G}$ is finite and upper semicontinuous.

The following lemma will be applied later to functions $\eta$ that form a partition
of unity in a $1$-quasiopen set.
Recall that we understand $N^{1,1}_0(H)$ to be a subspace of $N^{1,1}(X)$.

\begin{lemma}\label{lem:zero extension of Sobolev function}
	Let $H\subset X$ be $\mu$-measurable, let $u\in N^{1,1}(H)$ be bounded,
	and let $\eta\in N^{1,1}_0(H)$
	with $0\le\eta\le 1$ on $X$.
	Then $\eta u\in N^{1,1}_0(H)$ with a $1$-weak upper gradient
	$\eta g_{u}+|u|g_{\eta}$  (in $X$, with the interpretation that an undefined function
	times zero is zero).
\end{lemma}

\begin{proof}
We have $|u|\le M$ in $H$ for some $M\ge 0$.
By the Leibniz rule, see \cite[Theorem 2.15]{BB}, we know that $\eta u\in N^{1,1}(H)$ with a
$1$-weak upper gradient $\eta g_u+|u|g_\eta$ in $H$.
Moreover, $-M\eta\le \eta u\le M\eta\in N^{1,1}_0(H)$, and so by \cite[Lemma 2.37]{BB} we conclude that
$\eta u\in N_0^{1,1}(H)$, with $g_{\eta u}=0$ in $X\setminus H$ by
\eqref{eq:upper gradient in coincidence set}.
Finally, by \cite[Proposition 3.10]{BB-OD} we know
that $\eta g_u+|u|g_\eta$ (with $g_{\eta}=0$ a.e. in $X\setminus H$
by \eqref{eq:upper gradient in coincidence set}) is a $1$-weak upper gradient of $\eta u$
in $X$.
\end{proof}

Next we observe that convergence in the $\BV$ norm implies the following pointwise convergence; this
follows from \cite[Lemma 4.2]{LaSh}.

\begin{lemma}\label{lem:norm and pointwise convergence}
	Let $u_i,u\in\BV(X)$ with $u_i\to u$ in $\BV(X)$. By passing to a subsequence (not relabeled),
	we have $u_i^{\wedge}\to u^{\wedge}$ and $u_i^{\vee}\to u^{\vee}$ $\mathcal H$-a.e. in $X$.
\end{lemma}

We have the following result for $\BV$ functions whose variation
measure has no singular part; recall the decomposition
\eqref{eq:variation measure decomposition}.

\begin{theorem}\label{thm:BV with only abs cont part}
Let $\Om\subset X$ be open and let $v\in L^1_{\loc}(\Om)$ with
$\Vert Dv\Vert(\Om)<\infty$ and $\Vert Dv\Vert^s(U)=0$ for a
$\mu$-measurable set $U\subset\Om$.
Then a modification $\widehat{v}$ of $v$ in a $\mu$-negligible subset of
$U$ satisfies $\widehat{v}\in N^{1,1}_{\loc}(U)$ such that for every $\mu$-measurable
$H\subset U$,
\[
\int_{H}g_{\widehat{v}}\,d\mu\le C_0\Vert Dv\Vert(H)
\]
where $g_{\widehat{v}}$ is the minimal $1$-weak upper gradient of $\widehat{v}$ in $U$ and
$C_0\ge 1$ is a constant depending only on the doubling constant of $\mu$
and the constants in the Poincar\'e inequality.
\end{theorem}
\begin{proof}
This result is given in \cite[Theorem 4.6]{HKLL}, except that there
it is assumed that $v\in\BV(\Om)$ (that is, $v$ is in $L^1(\Om)$ and not just in
$L^1_{\loc}(\Om)$). However, by exhausting $\Om$ with relatively compact open
sets and applying \cite[Theorem 4.6]{HKLL} in these sets,
we obtain the result (note that by \eqref{eq:ae to qe}
and \eqref{eq:quasieverywhere equivalence classes}
we know that we do
not need to keep redefining $\widehat{v}$ in this construction).
\end{proof}

Finally, we have the following two simple results for $1$-quasiopen sets.

\begin{lemma}\label{lem:characteristic function of quasiopen set}
 Let $U\subset X$ be $1$-quasiopen. Then $\ch_{U}$ is $1$-quasi
lower semicontinuous.
\end{lemma}
\begin{proof}
Let $\eps>0$. We find an open set $G\subset X$ such that
$\capa_1(G)<\eps$ and $U\cup G$ is open.
Thus $U$ is open in the subspace topology of $X\setminus G$, and so
$\ch_{U}|_{X\setminus G}$ is lower semicontinuous.
\end{proof}

Conversely, it is easy to see that the super-level sets $\{u>t\}$, $t\in\R$,
of a $1$-quasi lower semicontinuous function $u$ are $1$-quasiopen;
see e.g. the proof of \cite[Proposition 3.4]{BBM}. We will use this
fact, or its analog for $1$-quasi (upper semi-)continuous functions,
without further notice.

\begin{lemma}\label{lem:absolute continuity}
	Let $U\subset X$ be $1$-quasiopen,
	let $v\in N^{1,1}(U)$ with $\Vert Dv\Vert(U)<\infty$, and let
	$A\subset U$ with $\mu(A)=0$. Then $\Vert Dv\Vert(A)=0$.
\end{lemma}

Note that $v\in N^{1,1}(U)$ does not automatically imply $\Vert Dv\Vert(U)<\infty$,
since the latter involves an extension to an open set.

\begin{proof}
	We find open sets $W_j\supset A$, $j\in\N$, such that $\mu(W_j)\to 0$.
	Then the sets $W_j\cap U$ are easily seen to be $1$-quasiopen,
	and so by Theorem \ref{thm:characterization of total variational} we get
	\[
	\Vert Dv\Vert(A)\le \Vert Dv\Vert (W_j\cap U)\le \int_{W_j\cap U}g_{v, W_j\cap U}\,d\mu
	\le \int_{W_j\cap U}g_v\,d\mu
	\to 0\quad \textrm{as }j\to\infty,
	\]
	where $g_v$ is the minimal $1$-weak upper gradient of $v$ in $U$.
\end{proof}

\section{The discrete convolution method}

In this section we study partitions of unity in $1$-quasiopen sets
and then we use these to develop the discrete convolution method
in such sets.
To construct the partitions of unity, we first need suitable cutoff functions in
quasiopen sets. These cannot be taken to be Lipschitz functions,
but we can use Newton-Sobolev functions instead.
The following definition and proposition are analogs of the theory in the
case $1<p<\infty$, which was studied in the metric setting in \cite{BBL-SS}.

\begin{definition}
	A set $A\subset D$ is a \emph{1-strict subset} of $D$ if there is a function
	$\eta\in N_0^{1,1}(D)$ such that $\eta=1$ in $A$.
	
	A countable family $\{U_j\}_{j=1}^{\infty}$ of $1$-quasiopen sets is a \emph{quasicovering} of a $1$-quasiopen set $U$ if
	$\bigcup_{j=1}^{\infty}U_j\subset U$ and
	$\capa_1\left(U\setminus \bigcup_{j=1}^{\infty}U_j\right)=0$.
	If every $U_j$ is a $1$-finely open $1$-strict subset of $U$ and
	$\overline{U_j}\Subset U$, then $\{U_j\}_{j=1}^{\infty}$ is a
	\emph{1-strict quasicovering} of $U$.
\end{definition}

\begin{proposition}[{\cite[Proposition 5.4]{L-NC}}]\label{prop:existence of strict quasicovering}
	If $U\subset X$ is $1$-quasiopen, then there exists a $1$-strict quasicovering
	$\{U_j\}_{j=1}^{\infty}$ of  $U$.
	Moreover, the associated Newton-Sobolev functions can be chosen compactly supported in $U$.
\end{proposition}

We will need $1$-strict quasicoverings with some additional properties.
In the next proposition, we adapt a quasicovering to a given $\BV$ function.
Recall that the class $N_c^{1,1}(U)$ consists of those
functions $u\in N^{1,1}(X)$ that have
compact support in $U$, i.e. $\supp u\Subset U$.

\begin{proposition}\label{prop:existence of suitable strict quasicovering}
Let $U\subset \Om\subset X$ be such that $U$ is $1$-quasiopen and
$\Om$ is open, and let $u\in L^1_{\loc}(\Om)$ with $\Vert Du\Vert(\Om)<\infty$.
Then there exists a $1$-strict quasicovering
	$\{U_j\}_{j=1}^{\infty}$ of $U$,
	and associated Newton-Sobolev functions $\{\rho_j\in N_c^{1,1}(U)\}_{j=1}^{\infty}$
	such that $-\infty<\inf_{\supp \rho_j}u^{\wedge}\le \sup_{\supp \rho_j}u^{\vee}<\infty$ for all $j\in\N$.
\end{proposition}
\begin{proof}
Define $V_j:=\{x\in\Om:\,-j<u^{\wedge}(x)\le u^{\vee}(x)<j\}$ for each $j\in\N$.
By Proposition \ref{prop:quasisemicontinuity}
and the fact that the intersection of two $1$-quasiopen sets is $1$-quasiopen
(see e.g. \cite[Lemma 2.3]{Fug}),
each of these sets is $1$-quasiopen.
By \cite[Lemma 3.2]{KKST3} we know that $\mathcal H \big(\Om\setminus \bigcup_{j=1}^{\infty}V_j\big)=0$.
For each $j\in\N$, apply Proposition \ref{prop:existence of strict quasicovering}
to find a $1$-strict quasicovering
	$\{U_{j,k}\}_{k=1}^{\infty}$ of  $V_j$, and
	the associated Newton-Sobolev functions
	$\rho_{j,k}\in N_c^{1,1}(V_j)$.
	Then $\{U_{j,k}\}_{j,k=1}^{\infty}$ is a $1$-strict quasicovering of $U$
	with the associated Newton-Sobolev functions
	$\rho_{j,k}\in N_c^{1,1}(U)$, such that
	\[
	-\infty<-j\le \inf_{\supp \rho_{j,k}}u^{\wedge}\le \sup_{\supp \rho_{j,k}}u^{\vee}\le j<\infty
	\]
	for all $j,k\in\N$.
\end{proof}

By truncating if necessary, we can always assume that the Newton-Sobolev functions
take values between $0$ and $1$.

Now we construct the partition of unity.

\begin{proposition}\label{prop:locally finite partition of unity}
Let $U\subset X$ be $1$-quasiopen and let $\{U_j\}_{j=1}^{\infty}$ be a
$1$-strict quasicovering of $U$
with the associated nonnegative Newton-Sobolev functions $\rho_j\in N_c^{1,1}(U)$.
Then we can find functions $\eta_j\in N_c^{1,1}(U)$ such that $\eta_1=\rho_1$,
$0\le \eta_j\le \rho_j$ for all $j\in\N$,
$\sum_{j=1}^{\infty}\eta_j=1$ $1$-q.e. in $U$,
and $1$-q.e. $x\in U$ has a $1$-fine neighborhood where $\eta_j\neq 0$
for only finitely many $j\in\N$.
\end{proposition}

We describe the last two conditions by saying that $\{\eta_j\}_{j=1}^{\infty}$
is a $1$-finely locally finite partition of unity on $U$.

\begin{proof}
Define recursively for each $j\in\N$
\[
\eta_j:=\min \Bigg\{\Big(1-\sum_{l=1}^{j-1}\eta_{l}\Big)_+,\rho_j\Bigg\}.
\]
It is clear that $0\le \eta_j\le \rho_j$ for all $j\in\N$, and then by the lattice
property \eqref{eq:Newton Sobolev functions form lattice}
we get $\eta_j\in N^{1,1}_c(U)$.
Moreover, for $1$-q.e. $x\in U$ there is $k\in\N$ such that
$x\in U_k$, and thus
$\sum_{j=1}^{k}\eta_j=1$ in $U_k$
and $\eta_j=0$ in $U_k$ for all $j\ge k+1$.
Thus  $\eta_j\neq 0$
for only finitely many $j\in\N$ in a $1$-fine neighborhood of $x$, and
$\sum_{j=1}^{\infty}\eta_j=1$ $1$-q.e. in $U$.
\end{proof}

\begin{remark}\label{rmk:discrete convolutions}
In an open set $\Om$, we can pick a \emph{Whitney covering} consisting of balls
$B_j=B(x_j,r_j)$ that have radius
comparable to the distance to $X\setminus \Om$, and then
we can pick a Lipschitz partition of unity $\{\eta_j\}_{j=1}^{\infty}$
subordinate to this covering. Then the discrete convolution approximation
of a function $u\in\BV(\Om)$ is defined by
\[
v:=\sum_{j=1}^{\infty}u_{B_j}\eta_j\in \liploc(\Om).
\]
Using the Poincar\'e inequality \eqref{eq:poincare}, it can be shown that
$v$ has a $1$-weak upper gradient of the form
\[
C\sum_{j=1}^{\infty}\frac{\Vert Du\Vert(B(x_j,5\lambda r_j))}{\mu(B(x_j,5\lambda r_j))}\ch_{B(x_j,r_j)},
\]
see e.g. the proof of \cite[Proposition 4.1]{KKST3}.
However, when $\{\eta_j\}_{j=1}^{\infty}$ is instead a partition of unity
in a $1$-quasiopen set, the situation is more complicated, in particular
because the Poincar\'e inequality is more difficult to apply.
For this reason, using integral averages like $u_{B_j}$ appears to be
too crude a method, and instead we will make use of
the preliminary approximation results and
other machinery developed in Section \ref{sec:preliminary results}.
\end{remark}

The following theorem gives the discrete convolution technique in $1$-quasiopen sets.

\begin{theorem}\label{thm:discrete conv in quasiopen set}
Let $U\subset \Om\subset X$ be such that $U$ is $1$-quasiopen and
$\Om$ is open, and let $u\in L^1_{\loc}(\Om)$ with $\Vert Du\Vert(\Om)<\infty$.
Let $0<\eps<1$.
Then we find a partition of unity $\{\eta_j\in N_c^{1,1}(U)\}_{j=1}^{\infty}$
in $U$ and functions $u_j\in N^{1,1}(\{\eta_j>0\})$ such that the function
\begin{equation}\label{eq:definition of v}
v:=\sum_{j=1}^{\infty}\eta_j u_j
\end{equation}
satisfies
$\Vert v-u\Vert_{L^1(U)}<\eps$,
$\int_U g_v\,d\mu<\Vert Du\Vert(U)+\eps$,
and
\begin{equation}\label{eq:L infinity bound}
\sup_{U}|v-u^{\vee}|\le 9\sup_{U} (u^{\vee}-u^{\wedge})+\eps.
\end{equation}
Moreover, understanding $v-u$ to be zero extended to $X\setminus U$, we have
\begin{equation}\label{eq:glued discrete convolution}
\Vert D(v-u)\Vert(X)< 2\Vert Du\Vert(U)+\eps\quad\textrm{and}\quad
\Vert D(v-u)\Vert(X\setminus U)=0,
\end{equation}
and $|v-u|^{\vee}=0$ $\mathcal H$-a.e. in $X\setminus U$.
\end{theorem}

Note that we may have $\sup_{U} (u^{\vee}-u^{\wedge})=\infty$ and then
\eqref{eq:L infinity bound} is vacuous.
The conditions $\Vert D(v-u)\Vert(X\setminus U)=0$ and
$|v-u|^{\vee}=0$ $\mathcal H$-a.e. in $X\setminus U$ essentially say that
$v$ and $u$ have the same ``boundary values''.
This is the crucial new property that we obtain compared with
Proposition \ref{prop:approximation with L infinity bound},
because it says that $v$ can always be ``glued'' nicely with $u$,
in the sense of \eqref{eq:w defined by pasting} in the next section.

\begin{proof}
First we choose a suitable partition of unity in $U$.
By Proposition \ref{prop:existence of suitable strict quasicovering}
we find  a $1$-strict quasicovering
	$\{\widetilde{U}_j\}_{j=1}^{\infty}$ of $U$,
	and associated Newton-Sobolev functions $\{\widetilde{\rho}_j\in N_c^{1,1}(U)\}_{j=1}^{\infty}$, $0\le \widetilde{\rho}_j\le 1$,
	such that
	\[
	-\infty<\inf_{\supp \widetilde{\rho}_{j}}u^{\wedge}\le \sup_{\supp \widetilde{\rho}_{j}}u^{\vee}<\infty
	\]
	for all $j\in\N$.
	Since $\widetilde{\rho}_j=1$ in the $1$-finely open set $\widetilde{U}_j$ for each $j\in\N$,
	we have $\bigcup_{j=1}^k \widetilde{U}_j\subset \fint\left\{\max_{j\in\{1,\ldots k\}}\widetilde{\rho}_j=1\right\}$
	for each $k\in\N$.
	Now by the fact that $\{\widetilde{U}_j\}_{j=1}^{\infty}$ is a $1$-quasicovering of $U$ and
	by Lemma \ref{lem:variation measure and capacity},
	\begin{align*}
	\Vert Du\Vert\Bigg(U\setminus \fint\Big\{\max_{j\in\{1,\ldots k\}}\widetilde{\rho}_j=1\Big\}\Bigg)
	&\le \Vert Du\Vert\Bigg(U\setminus \bigcup_{j=1}^k \widetilde{U}_j\Bigg)\\
	&\overset{k\to\infty}{\to}
	\Vert Du\Vert\left(U\setminus \bigcup_{j=1}^{\infty}\widetilde{U}_j\right)=0;
	\end{align*}
	note that 1-quasiopen sets are easily seen to be $\Vert Du\Vert$-measurable
	by using Lemma \ref{lem:variation measure and capacity}, see \cite[Lemma 3.5]{L-Leib}.
	Thus for some $N\in\N$, we have
	\[
	\Vert Du\Vert\Bigg(U\setminus \fint\Big\{\max_{j\in\{1,\ldots N\}}\widetilde{\rho}_j=1\Big\}\Bigg)<\frac{\eps}{8C_0},
	\]
	where $C_0$ is the constant from Theorem \ref{thm:BV with only abs cont part}.
	Now define $U_1:=\bigcup_{l=1}^{N}\widetilde{U}_l$,
	$U_j:=\widetilde{U}_{N-1+j}$ for $j=2,3,\ldots$,
	$\rho_1:=\max_{l\in\{1,\ldots N\}}\widetilde{\rho}_l$, and
	$\rho_j:=\widetilde{\rho}_{N-1+j}$ for $j=2,3,\ldots$.
	Then $\{U_j\}_{j=1}^{\infty}$ is another $1$-strict quasicovering of $U$
	with associated Newton-Sobolev functions $\rho_j\in N_c^{1,1}(U)$,
	such that $0\le \rho_j\le 1$ and
	\[
	-\infty< \inf_{\supp \rho_{j}}u^{\wedge}\le \sup_{\supp \rho_{j}}u^{\vee}<\infty
	\]
	for all $j\in\N$.
	Moreover,
	\[
	\Vert Du\Vert(U\setminus \fint\{\rho_1=1\})<\frac{\eps}{8C_0}.
	\]
	
	Then by Proposition \ref{prop:locally finite partition of unity}
we find a nonnegative, $1$-finely locally finite partition of unity
$\{\eta_j\in N^{1,1}_c(U)\}_{j=1}^{\infty}$ in $U$
such that
	\begin{equation}\label{eq:finiteness condition of partition}
	-\infty< \inf_{\supp \eta_{j}}u^{\wedge}\le \sup_{\supp \eta_{j}}u^{\vee}<\infty
	\end{equation}
	for all $j\in\N$.
Moreover, $\eta_1=\rho_1$ and so
\begin{equation}\label{eq:choice of fint eta 1}
\Vert Du\Vert(U\setminus \fint\{\eta_1=1\})<\frac{\eps}{8C_0}.
\end{equation}
(In the rest of the proof, any other partition of unity
satisfying the properties mentioned in this paragraph would also work.)

For each $j\in\N$, since we have $\supp \eta_j\Subset \Om$,
there exists an open set $\Om_j$ with
$\supp \eta_j\subset \Om_j\Subset \Om$, and then $u\in \BV(\Om_j)$.
Since every function $\eta_j\in N_c^{1,1}(U)\subset N^{1,1}(X)$
is $1$-quasicontinuous, every set $\{\eta_j>0\}$ is $1$-quasiopen.
Now by Proposition \ref{prop:approximation with L infinity bound}
we find sequences $(u_{j,i})\subset N^{1,1}(\{\eta_j>0\})$ such that
$u_{j,i}\to u$ in $L^1(\{\eta_j>0\})$,
\begin{equation}\label{eq:uji and u L infinity bound}
\sup_{\{\eta_j>0\}} |u_{j,i}-u^{\vee}|\le 9\sup_{\{\eta_j>0\}} (u^{\vee}-u^{\wedge})+\eps<\infty\ \textrm{ (by \eqref{eq:finiteness condition of partition})}
\end{equation}
for all $i\in\N$, and
\[
\lim_{i\to\infty}\int_{\{\eta_j>0\}} g_{u_{j,i}}\,d\mu=\Vert Du\Vert(\{\eta_j>0\}),
\]
where each $g_{u_{j,i}}$
denotes (here and later) the minimal $1$-weak upper gradient of $u_{j,i}$ in $\{\eta_j>0\}$.
By passing to subsequences (not relabeled), we can also assume that
$u_{j,i}\to u$ a.e. in $\{\eta_j>0\}$.
For any set $W\subset X$, the function $\ch_{\overline{W}^1}$
is $1$-quasi upper semicontinuous by Theorem \ref{thm:finely open is quasiopen and vice versa} and Lemma \ref{lem:characteristic function of quasiopen set}, and then the function $\eta_j\ch_{\overline{W}^1}$ is also
$1$-quasi upper semicontinuous.
Thus by Proposition \ref{prop:weak convergence with quasisemicontinuous test function} we get
\begin{equation}\label{eq:uji to u weak convergence fact 2}
\limsup_{i\to\infty}\int_{\{\eta_j>0\}\cap \overline{W}^1}\eta_j g_{u_{j,i}}\,d\mu
\le \int_{\{\eta_j>0\}\cap \overline{W}^1}\eta_j\,d\Vert Du\Vert
\end{equation}
for each $j\in\N$.
By a suitable choice of indices $i(j)\in\N$, for each $j\in\N$ we have
with $u_j:=u_{j,i(j)}$ that $u_j\in N^{1,1}(\{\eta_j>0\})$,
\begin{equation}\label{eq:uj and u L infinity bounds}
\sup_{\{\eta_j>0\}} |u_{j}-u^{\vee}|\le 9\sup_{\{\eta_j>0\}} (u^{\vee}-u^{\wedge})
+\eps<\infty,
\end{equation}
	\begin{equation}\label{eq:choice of ujs and uniform convergence}
	\Vert u_{j}-u\Vert_{L^1(\{\eta_j>0\})}<2^{-j}\eps,\quad\textrm{and}\quad\int_{\{\eta_j>0\}} |u_j-u|g_{\eta_j}\,d\mu<\frac{2^{-j-2}\eps}{C_0},
	\end{equation}
	where the last inequality is achieved by Lebesgue's dominated convergence
	theorem, exploiting the boundedness of $\sup_{\{\eta_j>0\}} |u_{j,i}-u^{\vee}|$.
	Here $g_{\eta_j}$ is the minimal $1$-weak upper gradient of $\eta_j$
	in $X$.
	Define $W_0:=X$, $W_1:=X\setminus \{\eta_1=1\}$,
	and
	\[
	W_k:=X\setminus  \bigcup_{j=1}^k\supp \eta_j\quad\textrm{for }k=2,3,\ldots.
	\]
	By \eqref{eq:uji to u weak convergence fact 2} we can also assume for each $j\in\N$
	\begin{equation}\label{eq:choice of ujs and weak convergence 2}
	\int_{\{\eta_j>0\}\cap \overline{W_k}^1}\eta_j g_{u_{j}}\,d\mu
< \int_{\{\eta_j>0\}\cap \overline{W_k}^1}\eta_j\,d\Vert Du\Vert+\frac{2^{-j-2}\eps}{C_0}
	\end{equation}
	for the (finite number of) choices $k=0,\ldots,j$.
	Using Lebesgue's dominated convergence theorem as above, we have
	\begin{equation}\label{eq:uji to u fact}
	\lim_{i\to\infty}\int_{\{\eta_j>0\}}|u_j-u_{j,i}|g_{\eta_{j}}\,d\mu
= \int_{\{\eta_j>0\}}|u_j-u|g_{\eta_{j}}\,d\mu.
\end{equation}
By the definition of $v$ given in \eqref{eq:definition of v}
and by \eqref{eq:uj and u L infinity bounds},
\eqref{eq:choice of ujs and uniform convergence},
we clearly have
\[
\sup_{U}|v-u^{\vee}|\le 9\sup_{U} (u^{\vee}-u^{\wedge})+\eps
\]
and
\[
	\Vert v-u\Vert_{L^1(U)}
	=\Vert \sum_{j=1}^{\infty}\eta_j(u_{j}-u)\Vert_{L^1(U)}
	\le \sum_{j=1}^{\infty}\Vert u_{j}-u\Vert_{L^1(\{\eta_j>0\})}<\eps,
\]
as desired. Similarly,
\[
	\big\Vert \sum_{j=1}^{\infty}\eta_j |u_j-u|\big\Vert_{L^1(U)}
	\le \sum_{j=1}^{\infty}\Vert u_{j}-u\Vert_{L^1(\{\eta_j>0\})}<\eps,
\]
and so in particular $\sum_{j=1}^{\infty}\eta_j |u_j-u|\in L^1(U)$
and thus we have
\begin{equation}\label{eq:L1 convergence of partial sums}
\sum_{j=1}^{l}\eta_{j}(u_{j}-u)\to v-u\quad\textrm{in }L^1(U)\textrm{ as } l\to\infty
\end{equation}
by Lebesgue's dominated convergence theorem.

Moreover, for every $j\in\N$ we have
$\eta_j(u_j-u_{j,i})\to \eta_j(u_j-u)$ in $L^1(X)$.
By \eqref{eq:finiteness condition of partition},
\eqref{eq:uji and u L infinity bound}, and \eqref{eq:uj and u L infinity bounds},
we know that
\begin{equation}\label{eq:uji and uj are bounded}
u_{j,i}\textrm{ and }u_j\textrm{ are bounded in }\{\eta_j>0\}. 
\end{equation}
Thus by the lower semicontinuity of the total variation with respect
to $L^1$ convergence and by
Lemma \ref{lem:zero extension of Sobolev function},
we get for any open set $W\subset X$ (in fact any $1$-quasiopen set, see
comment below)
\begin{equation}\label{eq:sum eta j uj u}
	\begin{split}
&\Vert D\big(\eta_{j}(u_{j}-u)\big)\Vert(W)
\le \liminf_{i\to\infty}\int_W g_{\eta_{j}(u_j-u_{j,i})}\,d\mu\\
&\qquad\qquad  \le \liminf_{i\to\infty}\left(\int_{\{\eta_j>0\}}|u_j-u_{j,i}|g_{\eta_{j}}\,d\mu
+\int_{W\cap \{\eta_j>0\}}\eta_j(g_{u_j}+g_{u_{j,i}})\,d\mu\right)\\
&\qquad\qquad  \le \int_{\{\eta_j>0\}}|u_j-u|g_{\eta_{j}}\,d\mu
+\int_{W\cap \{\eta_j>0\}}\eta_jg_{u_j}\,d\mu
+\int_{\overline{W}^1\cap \{\eta_j>0\}}\eta_j\,d\Vert Du\Vert
\end{split}
	\end{equation}
by \eqref{eq:uji to u fact} and \eqref{eq:uji to u weak convergence fact 2}.
Note that with $W=X$, all the terms on the right-hand side are finite, and so
$\Vert D\big(\eta_{j}(u_{j}-u)\big)\Vert(X)<\infty$ and then
by Theorem \ref{thm:characterization of total variational} the above holds also for
$1$-quasiopen $W$.
For $k\in\N$,
note that
\begin{equation}\label{eq:eta j zero in Wk}
\eta_j=0\ \ \textrm{for }j=1,\ldots k-1\ \textrm{ in }W_k
\end{equation}
and that the set $W_1$ is $1$-quasiopen
by the quasicontinuity of $\eta_1$, while the sets $W_2,W_3,\ldots$ are open.
	Using \eqref{eq:BV functions form vector space}, we get for all
	$k,l\in\N$, $l\ge k$,
	\begin{equation}\label{eq:sum up to i is in BV X}
	\begin{split}
	&\Vert D\Big(\sum_{j=1}^{l}\eta_{j}(u_{j}-u)\Big)\Vert(W_k)
	\le \sum_{j=1}^{l}\Vert D(\eta_{j}(u_{j}-u))\Vert(W_k)\\
	&= \sum_{j=k}^{l}\Vert D(\eta_{j}(u_{j}-u))\Vert(W_k)\quad\textrm{by \eqref{eq:eta j zero in Wk}}\\
	&\le \sum_{j=k}^{l}\int_{\{\eta_j>0\}} |u_j-u|g_{\eta_j}\,d\mu
	+\sum_{j=k}^{l}\int_{W_k\cap \{\eta_j>0\}} \eta_j g_{u_j}\,d\mu\\
	&\qquad\qquad +\sum_{j=k}^{l}\int_{\overline{W_k}^1\cap \{\eta_j>0\}} \eta_j\,d\Vert Du\Vert\quad\textrm{by }\eqref{eq:sum eta j uj u}\\
	&< \frac{1}{C_0}\sum_{j=k}^{l} 2^{-j-2}\eps
	+2\sum_{j=k}^{l}\int_{\overline{W_k}^1\cap \{\eta_j>0\}} \eta_j\,d\Vert Du\Vert
	+\frac{1}{C_0}\sum_{j=k}^{l} 2^{-j-2}\eps\quad\textrm{by }\eqref{eq:choice of ujs and uniform convergence},\eqref{eq:choice of ujs and weak convergence 2}\\
	&= 2\sum_{j=k}^{l}\int_{\overline{W_k}^1\cap \{\eta_j>0\}} \eta_j\,d\Vert Du\Vert
	+\frac{2^{-k}\eps}{C_0}.
	\end{split}
	\end{equation}
	For $k=0$ (recall that $W_0=X$) and any $1\le m\le l$ we get by essentially the same calculation
	\begin{equation}\label{eq:sum from m to l is in BV X}
	\Vert D\Big(\sum_{j=m}^{l}\eta_{j}(u_{j}-u)\Big)\Vert(X)
	< 2\sum_{j=m}^{l}\int_{U} \eta_j\,d\Vert Du\Vert+2^{-m}\eps.
	\end{equation}
	By \eqref{eq:L1 convergence of partial sums} we had $\sum_{j=1}^{l}\eta_{j}(u_{j}-u)\to v-u$
	in $L^1(U)$, so understanding
	$v-u$ to be zero extended to $X\setminus U$, we now get
	by lower semicontinuity of the total variation with respect to $L^1$-convergence,
	\begin{align*}
	\Vert D(v-u)\Vert (X)
	\le \liminf_{l\to\infty}
	\Vert D\Big(\sum_{j=1}^{l}\eta_{j}(u_{j}-u)\Big)\Vert(X)
	&\le 2\sum_{j=1}^{\infty}\int_{U} \eta_j\,d\Vert Du\Vert+2^{-1}\eps\\
	&= 2\Vert Du\Vert(U)+2^{-1}\eps,
	\end{align*}
	proving the first inequality in \eqref{eq:glued discrete convolution}.
	Now by Theorem \ref{thm:lower semic in quasiopen sets} and
	\eqref{eq:sum up to i is in BV X},
	we have for each $k\in\N$
	\begin{equation}\label{eq:D v minus u in Wk}
	\begin{split}
	\Vert D(v-u)\Vert(W_k)
	&\le \liminf_{l\to\infty}\Vert D\Big(\sum_{j=1}^{l}\eta_{j}(u_{j}-u)\Big)\Vert(W_k)\\
	&\le 2\sum_{j=k}^{\infty}\int_{\overline{W_k}^1\cap \{\eta_j>0\}} \eta_j\,d\Vert Du\Vert
	+\frac{2^{-k}\eps}{C_0}\\
	&\le 2\sum_{j=k}^{\infty}\int_{U\setminus \fint\{\eta_1=1\}} \eta_j\,d\Vert Du\Vert
	+\frac{2^{-k}\eps}{C_0}.
	\end{split}
	\end{equation}
	Note that
	$\sum_{j=k}^{\infty}\eta_j\to 0$ $1$-q.e. in $U$ as $k\to\infty$,
	and then also $\Vert Du\Vert$-a.e. in $U$ by Lemma
	\ref{lem:variation measure and capacity}.
	Since $W_k\supset X\setminus U$ for all $k\in\N$,
	by Lebesgue's dominated
	convergence theorem we now get $\Vert D(v-u)\Vert(X\setminus U)=0$,
	proving the second inequality in \eqref{eq:glued discrete convolution}.
	
	Moreover, $h_l:=\sum_{j=1}^{l}\eta_{j}(u_{j}-u)$ is a Cauchy
	sequence in $\BV(X)$, since
	by \eqref{eq:sum from m to l is in BV X} we get for any $1\le m<l$
	\[
	\Vert D\Big(\sum_{j=m}^{l}\eta_{j}(u_{j}-u)\Big)\Vert(X)
	< 2\sum_{j=m}^{\infty}\int_{U} \eta_j\,d\Vert Du\Vert+2^{-m}\eps
	\to 0
	\]	
	as $m\to\infty$. Thus $h_l\to v-u$ in $\BV(X)$ (and not just in $L^1(X)$
	as noted in \eqref{eq:L1 convergence of partial sums}).
	Since each $h_l$ has compact support in $U$ and thus
	$h_l^{\wedge}=0=h_l^{\vee}$ in $X\setminus U$,
	by Lemma \ref{lem:norm and pointwise convergence}
	it follows that $(v-u)^{\wedge}(x)=0=(v-u)^{\vee}(x)$ for
	$\mathcal H$-a.e. $x\in X\setminus U$, and so also
	$|v-u|^{\vee}(x)=0$ for $\mathcal H$-a.e. $x\in X\setminus U$,
	as desired.
	
	Since the partition of unity
	$\{\eta_j\}_{j=1}^{\infty}$ is $1$-finely locally finite,
	the sets $V_k:=\fint \left\{\sum_{j=1}^k \eta_j=1\right\}$ cover $1$-quasi all of $U$.
	Moreover, $v\in N^{1,1}(V_k)$ for all $k\in\N$; this follows from
	the fact that $v$ in $V_k$ is the finite sum $\sum_{j=1}^k \eta_j u_j$,
	which is in $N^{1,1}(X)$ by Lemma \ref{lem:zero extension of Sobolev function} and
	\eqref{eq:uji and uj are bounded}.
	Let $A\subset U$ such that $\mu(A)=0$.
	By Theorem \ref{thm:finely open is quasiopen and vice versa},
	each $V_k$ is $1$-quasiopen and then by
	Lemma \ref{lem:absolute continuity}
	we have $\Vert Dv\Vert(A\cap V_k)=0$ for all $k\in\N$
	(note that $\Vert Dv\Vert(\Om)<\infty$ by the first inequality in
	\eqref{eq:glued discrete convolution}, understanding $v$ to be extended
	to $\Om\setminus U$ as $u$).
	Thus using also Lemma \ref{lem:variation measure and capacity},
	\[
	\Vert Dv\Vert(A)\le \Vert Dv\Vert\left(A\cap \bigcup_{k=1}^{\infty}V_k\right)+
	\Vert Dv\Vert\left(A\setminus \bigcup_{k=1}^{\infty}V_k\right)=0.
	\]
	Thus $\Vert Dv\Vert$ is absolutely continuous with respect to $\mu$ in $U$,
	and so by Theorem \ref{thm:BV with only abs cont part} we know that
a modification $\widehat{v}$ of $v$ in a $\mu$-negligible subset of $U$
satisfies 	
	$\widehat{v}\in N^{1,1}_{\loc}(U)$ such that
	$\int_H g_{\widehat{v}}\,d\mu\le C_0\Vert Dv\Vert(H)$ for every $\mu$-measurable $H\subset U$,
	where $g_{\widehat{v}}$ is the minimal $1$-weak upper gradient of $\widehat{v}$ in $U$.
	Now for each $k\in\N$, $v$ and $\widehat{v}$ are both $1$-quasicontinuous on
	the $1$-quasiopen set $V_k$ by \eqref{eq:quasicontinuous on quasiopen set},
	with $v=\widehat{v}$ a.e. in $V_k$,
	and so by \eqref{eq:ae to qe} we have in fact $v=\widehat{v}$ $1$-q.e. in $V_k$.
	Thus $v=\widehat{v}$ $1$-q.e. in $U$ and
	then by \eqref{eq:quasieverywhere equivalence classes}
	we can in fact let $\widehat{v}=v$ everywhere in $U$.
	
	By \cite[Proposition 3.5]{BB-OD} and
	\cite[Remark 3.5]{S2} we know that $g_{v,\{\eta_1>0\}}=g_{v}$
	a.e. in $\{\eta_1>0\}$, that is, it does not make a difference whether
	we consider the minimal $1$-weak upper gradient of $v$ in $U$ or
	in the smaller $1$-quasiopen set $\{\eta_1>0\}$. Then by \eqref{eq:upper gradient in coincidence set} we have $g_{u_1}=g_{v,\{\eta_1>0\}}=g_v$ a.e. in
	$\{\eta_1=1\}$.
	It follows that
	\begin{align*}
	\int_U g_v\,d\mu
	&\le \int_{\{\eta_1=1\}} g_v\,d\mu+\int_{U\setminus \{\eta_1=1\}} g_v\,d\mu\\
	&\le \int_{\{\eta_1=1\}} g_{u_1}\,d\mu
	+C_0\Vert Dv\Vert(U\cap W_{1})\\
	&\le \int_{\{\eta_1=1\}} g_{u_1}\,d\mu
	+C_0\Vert D(v-u)\Vert(W_{1})
	+C_0\Vert Du\Vert(U\cap W_{1})\\
	&\le \int_{\{\eta_1=1\}} g_{u_1}\,d\mu
	+3C_0\Vert Du\Vert(U\setminus \fint\{\eta_1=1\})+2^{-1}\eps\quad\textrm{by }\eqref{eq:D v minus u in Wk}\\
	&< \Vert Du\Vert(U)+\eps
	\end{align*}
	by \eqref{eq:choice of ujs and weak convergence 2}
	with the choices $W_k=X$ and $j=1$,
	and \eqref{eq:choice of fint eta 1}.
\end{proof}

\begin{remark}
Note that in the usual discrete convolution technique described in
Remark \ref{rmk:discrete convolutions}, we only get the estimate
$\int_{\Om}g_v\,d\mu\le C\Vert Du\Vert(\Om)$ for some constant $C\ge 1$
depending on the doubling and Poincar\'e constants, whereas in
Theorem \ref{thm:discrete conv in quasiopen set} we obtained
$\int_{\Om}g_v\,d\mu\le \Vert Du\Vert(\Om)+\eps$.
Thus our technique may seem to be an improvement on the usual discrete convolution
technique already in open sets, but in fact the (usual)
discrete convolutions have other good properties, in particular
the uniform integrability of the upper gradients in the case where
$\Vert Du\Vert$ is absolutely continuous with respect to $\mu$, see 
\cite[Lemma 6]{FHK}.
The uniform integrability seems more difficult to obtain in the quasiopen case,
but it is also perhaps not interesting for the following reason:
if $\Vert Du\Vert$ is absolutely continuous
in a $1$-quasiopen set $U$, then Theorem \ref{thm:BV with only abs cont part}
(whose proof is based on discrete convolutions) already tells that
$u\in N_{\loc}^{1,1}(U)$, and so it is not interesting to approximate
$u$ with functions $v\in N_{\loc}^{1,1}(U)$ given by Theorem
\ref{thm:discrete conv in quasiopen set}.
\end{remark}

\section{An approximation result}

In this section we apply the discrete convolution technique of the previous section
to prove a new approximation result for $\BV$ functions, given in
Theorem \ref{thm:main approximation theorem} of the introduction.
In this result we approximate a $\BV$ function in the $\BV$ and $L^{\infty}$
norms by $\BV$ functions whose jump sets are of finite Hausdorff measure.

First we note that without the requirement of approximation in the $L^{\infty}$ norm, 
the theorem could be
proved by using standard discrete convolutions.
Indeed, if $\Om\subset X$ is an open set and $u\in\BV(\Om)$,
we can take a suitable open set $W\subset \Om$ containing the part of the jump
set $S_u=\{u^{\vee}>u^{\wedge}\}=\{x\in\Om:\,u^{\vee}(x)>u^{\wedge}(x)\}$ where the size of the jump $u^{\vee}-u^{\wedge}$ is small, and then
we can take a discrete convolution of $u$
in $W$. By gluing this with the function $u$ in $\Om\setminus W$, we get the desired approximation; we omit the details but the essential aspects of this kind of technique
are given in \cite[Corollary 3.6]{LaSh}. However, the open set $W$ may
unavoidably contain also large jumps of $u$, and so it seems
impossible to obtain approximation in the $L^{\infty}$ norm with this method.
We sketch this problem in the following example.

\begin{example}\label{ex:dense jump set}
Let $X=\R^2$ (unweighted) and $\Om:=(0,1)\times (0,1)$. Define the strips
\[
A_j:=\{x=(x_1,x_2)\in\R^2:\, 2^{-j}\le x_1< 2^{-j+1},\, 0< x_2< 1\},\quad j\in\N,
\]
and the function
\[
v:=\sum_{j=1}^{\infty}j^{-1}\ch_{A_j}\in\BV(\Om).
\]
Also take a function $w\in\BV(\Om)$, $0\le w\le 1$, for which
\[
S_w=\{x\in\Om:\,w^{\vee}(x)-w^{\wedge}(x)= 1\}
\]
is dense in $\Om$; we do not present the construction of such a function
but it can be taken to be the characteristic function of a suitable (fat)
Cantor-type set.

Then let $u:=v+w\in\BV(\Om)$.
Denote by $\mathcal H^1$ the $1$-dimensional Hausdorff measure; note that this
is comparable to the codimension one Hausdorff measure.
Since $\mathcal H^{1}(S_v)=\infty$
and $\mathcal H^{1}(S_w)<\infty$
(otherwise $\Vert Dw\Vert(\Om)=\infty$ by \eqref{eq:variation measure decomposition}), clearly $\mathcal H^{1}(S_u)=\infty$.
Suppose we take an open set $W\subset\Om$ containing the
set $\{u^{\vee}-u^{\wedge}<\delta\}$ for some (small) $\delta>0$.
Then $W$ is nonempty and so contains a point $x\in \{w^{\vee}-w^{\wedge}= 1\}$,
and then clearly also $x\in \{u^{\vee}-u^{\wedge}\ge 1/2\}$.
If $h$ is a continuous function in $W$ (for example if $h$ is
a discrete convolution of $u$),
then it is straightforward to check that $\Vert h-u\Vert_{L^{\infty}(W)}\ge 1/4$
and thus we do not have approximation in the $L^{\infty}$ norm.
\end{example}

To prove the approximation result, we need the following lemma;
recall the definition of the measure-theoretic interior from
\eqref{eq:measure theoretic interior}.

\begin{lemma}\label{lem:measure theoretic interior of a quasiopen set}
Let $U\subset X$ be $1$-quasiopen. Then $\mathcal H(U\setminus I_U)=0$.
\end{lemma}
\begin{proof}
By Theorem \ref{thm:finely open is quasiopen and vice versa}
we find a $1$-finely open set $V\subset U$ such that
$\mathcal H(U\setminus V)=0$. By \eqref{eq:thinness and measure thinness},
$V\subset I_V$,
and then obviously $V\subset I_U$.
\end{proof}

First we give the approximation result in the following form containing
more information than Theorem \ref{thm:main approximation theorem}.
The symbol $\Delta$ denotes the symmetric difference.

\begin{theorem}\label{thm:approximation}
Let $\Om\subset X$ be open, let $u\in\BV(\Om)$, and let $\eps,\delta>0$. Then
we find $w\in \BV(\Om)$ such that $\Vert w-u\Vert_{L^1(\Om)}<\eps$,
\[
\Vert D(w-u)\Vert(\Om)
< 2\Vert Du\Vert(\{0<u^{\vee}-u^{\wedge}< \delta\})+\eps,
\]
$\Vert w-u\Vert_{L^{\infty}(\Om)}\le 10\delta$,
$\mathcal H(S_{w}\Delta \{u^{\vee}-u^{\wedge}\ge \delta\})=0$,
and
\begin{equation}\label{eq:Lusin estimate 1}
\Vert Du\Vert(\{|w-u|^{\vee}\neq 0\})<\Vert Du\Vert(\{0<u^{\vee}-u^{\wedge}< \delta\})
+\eps
\end{equation}
and
\begin{equation}\label{eq:Lusin estimate 2}
\mu(\{|w-u|^{\vee}\neq 0\})<\eps.
\end{equation}
\end{theorem}

\begin{proof}
Take an open set $W$ such that
$\{0<u^{\vee}-u^{\wedge}< \delta\}\subset W\subset \Om$,
\[
\Vert Du\Vert(W)<\Vert Du\Vert(\{0<u^{\vee}-u^{\wedge}< \delta\})+\eps/4,
\]
and $\mu(W)<\eps$; recall that by the decomposition
\eqref{eq:variation measure decomposition}, the jump set $S_u$ is $\sigma$-finite
with respect to $\mathcal H$ and thus $\mu(S_u)=0$.
By Proposition \ref{prop:quasisemicontinuity},
the set $\{u^{\vee}-u^{\wedge}< \delta\}$ is $1$-quasiopen,
and then so is $U:=W\cap \{u^{\vee}-u^{\wedge}< \delta\}$.
Moreover,
\begin{equation}\label{eq:total variation in U}
\Vert Du\Vert(U)\le \Vert Du\Vert(W)
<\Vert Du\Vert(\{0<u^{\vee}-u^{\wedge}< \delta\})+\eps/4.
\end{equation}
By Theorem \ref{thm:discrete conv in quasiopen set} we find
a function $v\in N^{1,1}(U)$
satisfying $\Vert v-u\Vert_{L^1(U)}<\eps$,
\[
\sup_{U}|v-u^{\vee}|\le 9\sup_{U} (u^{\vee}-u^{\wedge})+\delta\le 10\delta,
\]
and,  understanding $v-u$ to be zero extended to $X\setminus U$,
\begin{equation}\label{eq:D v minus u in proof}
\Vert D(v-u)\Vert(X)< 2\Vert Du\Vert(U)+\eps/2.
\end{equation}
By Lebesgue's differentiation theorem, now also
$\Vert v-u\Vert_{L^{\infty}(U)}\le 10\delta$.
Define
\begin{equation}\label{eq:w defined by pasting}
w:=
\begin{cases}
v & \textrm{in }U,\\
u & \textrm{in }\Om\setminus U.
\end{cases}
\end{equation}
Then $\Vert w-u\Vert_{L^1(\Om)}<\eps$
and $\Vert w-u\Vert_{L^{\infty}(\Om)}\le 10\delta$.
From \eqref{eq:D v minus u in proof}, \eqref{eq:total variation in U} we get
\[
\Vert D(w-u)\Vert(\Om)< 2\Vert Du\Vert(U)+\eps/2
<2\Vert Du\Vert(\{0<u^{\vee}-u^{\wedge}< \delta\})+\eps,
\]
as desired.
The function $v$ is $1$-quasicontinuous on the $1$-quasiopen set $U$
by \eqref{eq:quasicontinuous on quasiopen set}, and then also $1$-finely continuous $1$-q.e. in $U$
by Theorem \ref{thm:fine continuity and quasicontinuity equivalence}.
By Lemma \ref{lem:measure theoretic interior of a quasiopen set} we also have
$x\in I_U$ for $\mathcal H$-a.e. $x\in U$.
By \eqref{eq:null sets of Hausdorff measure and capacity},
$\mathcal H$-a.e. $x\in U$ satisfies both these properties, and then
by \eqref{eq:thinness and measure thinness} we find that
$w^{\wedge}(x)=w^{\vee}(x)$.
Thus $\mathcal H(S_w\cap U)=0$.

By definition of $U$ we have
$\{u^{\vee}-u^{\wedge}\ge \delta\}=S_u\setminus U$.
Since $|w-u|^{\vee}=0$ $\mathcal H$-a.e. in $\Om\setminus U$
by Theorem \ref{thm:discrete conv in quasiopen set},
we have $u^{\wedge}=w^{\wedge}$ and $u^{\vee}=w^{\vee}$
$\mathcal H$-a.e. in $\Om\setminus U$, and so the sets
$\{u^{\vee}-u^{\wedge}\ge \delta\}$ and $S_w\setminus U$
coincide outside a $\mathcal H$-negligible set.
In total, $\mathcal H(S_w\Delta \{u^{\vee}-u^{\wedge}\ge \delta\})=0$, as desired.

Since $|w-u|^{\vee}=0$ $\mathcal H$-a.e. in
$\Om\setminus U$, this holds also $\mu$-a.e. and $\Vert Du\Vert$-a.e. in $\Om\setminus U$
(recall \eqref{eq:null sets of Hausdorff measure and capacity} and Lemma
\ref{lem:variation measure and capacity}). Thus we get estimates \eqref{eq:Lusin estimate 1}
and \eqref{eq:Lusin estimate 2}.
\end{proof}

\begin{proof}[Proof of Theorem \ref{thm:main approximation theorem}]
For each $i\in\N$, choose the function $u_i$ to be
$w\in\BV(\Om)$ as given by
Theorem \ref{thm:approximation} with the choices $\eps=1/i$ and $\delta=1/i$.
Then $\Vert u_i-u\Vert_{L^1(\Om)}<1/i$ and
\[
\Vert D(u_i-u)\Vert(\Om)
< 2\Vert Du\Vert(\{0<u^{\vee}-u^{\wedge}< 1/i\})+1/i\to 0\quad\textrm{as }i\to\infty,
\]
and so $\Vert u_i-u\Vert_{\BV(\Om)}\to 0$ as $i\to\infty$.
Also, $\Vert u_i-u\Vert_{L^{\infty}(\Om)}\le 10/i\to 0$ as desired.
By the decomposition \eqref{eq:variation measure decomposition}
we find that $\mathcal H(\{u^{\vee}-u^{\wedge}\ge 1/i\})<\infty$ for all $i\in\N$
and so 
\[
\mathcal H(S_{u_i})=\mathcal H(\{u^{\vee}-u^{\wedge}\ge 1/i\})<\infty
\]
for all $i\in\N$.
\end{proof}

We observe that the proofs of Theorems \ref{thm:approximation} and
\ref{thm:main approximation theorem} were quite straightforward, because most of the hard
work was already done in the proof of the discrete convolution technique,
Theorem \ref{thm:discrete conv in quasiopen set}. Since
Theorem \ref{thm:discrete conv in quasiopen set} can be applied rather easily
in any $1$-quasiopen set, we expect that it will be useful also in the context
of other problems,
for example if one considers minimization problems in $1$-quasiopen domains.

We say that $u\in \BV(\Om)$ is a \emph{special function of bounded variation}, and
denote $u\in\SBV(\Om)$, if the Cantor part of the variation measure vanishes, i.e.
 $\Vert Du\Vert^c(\Om)=0$.
The following approximation result was proved (with some more details)
in \cite[Corollary 5.15]{L-Appr}.

\begin{theorem}
Let $\Om\subset X$ be open and let $u\in\BV(\Om)$. Then there exists a
sequence $(u_i)\subset \SBV(\Om)$
such that
\begin{itemize}
\item $u_i\to u$ in $L^1(\Om)$ and $\Vert Du_i\Vert(\Om)\to \Vert Du\Vert(\Om)$,
\item $\lim_{i\to\infty}\Vert D(u_i-u)\Vert(\Om)= 2\Vert Du\Vert^c(\Om)$,
\item $\limsup_{i\to\infty}\Vert Du\Vert(\{|u_i-u|^{\vee}\neq 0\})
\le \Vert Du\Vert^c(\Om)$ and\\
$\lim_{i\to\infty}\mu(\{|u_i-u|^{\vee}\neq 0\})=0$,
\item $\lim_{i\to\infty}\Vert u_i-u\Vert_{L^{\infty}(\Om)}=0$, and
\item $\mathcal H(S_{u_i}\setminus S_u)=0$ for all $i\in\N$.
\end{itemize}
\end{theorem}

Combining this with
Theorem \ref{thm:approximation}, we get the following corollary.

\begin{corollary}\label{cor:approximation combined}
Let $\Om\subset X$ be open and let $u\in\BV(\Om)$. Then there exists a
sequence $(u_i)\subset \SBV(\Om)$ with $\mathcal H(S_{u_i})<\infty$
for all $i\in\N$, such that
\begin{itemize}
\item $u_i\to u$ in $L^1(\Om)$ and $\Vert Du_i\Vert(\Om)\to \Vert Du\Vert(\Om)$,
\item $\lim_{i\to\infty}\Vert D(u_i-u)\Vert(\Om)= 2\Vert Du\Vert^c(\Om)$,
\item $\limsup_{i\to\infty}\Vert Du\Vert(\{|u_i-u|^{\vee}\neq 0\})
\le \Vert Du\Vert^c(\Om)$ and\\
$\lim_{i\to\infty}\mu(\{|u_i-u|^{\vee}\neq 0\})=0$, and
\item $\lim_{i\to\infty}\Vert u_i-u\Vert_{L^{\infty}(\Om)}=0$.
\end{itemize}
\end{corollary}

The first condition in the corollary is often expressed by saying that the
$u_i$'s converge to $u$ in the \emph{strict sense}, whereas
the second condition
describes closeness in the $\BV$ norm.
The third condition describes approximation in the Lusin sense.
In all, the corollary states that we can always approximate a $\BV$ function in a rather
strong sense with functions that have neither a Cantor part of the variation measure
nor a large jump set.

\paragraph{Acknowledgments.}Part of the research for this paper was conducted
while the author was visiting Aalto University and the University of Cincinnati;
he wishes to thank Juha Kinnunen and Nageswari Shan\-mu\-ga\-lingam for the kind invitations.

\noindent Address:\\

\noindent Institut f\"ur Mathematik\\
Universit\"at Augsburg\\
Universit\"atsstr. 14\\
86159 Augsburg, Germany\\
E-mail: {\tt panu.lahti@math.uni-augsburg.de}

\end{document}